\documentclass[12pt]{article}
\pdfoutput=1

\usepackage{imsart}
\usepackage{amsmath} 
\usepackage{amsthm} 
\usepackage{hyperref} 
\usepackage{amsfonts}
\usepackage{tikz}
\usepackage{amstext}
\usepackage{amssymb}
\usepackage{paralist}

\newtheorem{theorem}{{Theorem}}[section]
\newtheorem{lemma}[theorem]{{Lemma}}
\newtheorem{proposition}[theorem]{{Proposition}}
\newtheorem{remark}[theorem]{{Remark}}
\newtheorem{definition}[theorem]{{Definition}}

\numberwithin{equation}{section} 

\begin{document}

\begin{frontmatter}

\title{Using Householder Matrices to Establish Mixing Test Critical Values}

\begin{aug}
	\author{Aaron Carl Smith}\\
	\href{mailto: aaron.smith@ucf.edu}{aaron.smith@ucf.edu}

	\address{Department of Mathematics \\ University of Central Florida \\ 4000 Central Florida Blvd \\ P.O. Box 161364 \\ Orlando, FL 32816-1364 \\ (407) 823-0538 \\ (407) 823-6253 (fax)}

\end{aug}

\begin{abstract}
A measure-preserving dynamical system can be approximated by a Markov shift with a bistochastic matrix.  This leads to using empirical stochastic matrices to measure and estimate properties of stirring protocols.  Specifically, the second largest eigenvalue can be used to statistically decide if a stirring protocol is weak-mixing, ergodic, or nonergodic.  Such hypothesis tests require appropriate probability distributions.  In this paper, we propose using Monte Carlo empirical probability distributions from unistochastic matrices to establish critical values.  These unistochastic matrices arise from randomly constructed Householder matrices.
\end{abstract}



\tableofcontents

\end{frontmatter}

\section{Introduction}

If a dynamical system has a probability measure and the system is measure-preserving, then partitioning the domain into $n$ states of equal measure leads to a Markov shift whose measure is defined by a bistochastic matrix and the length $n$ row vector
\begin{align} \begin{pmatrix} \tfrac{1}{n} \ldots \tfrac{1}{n} \end{pmatrix}.
\end{align}
This partition approximation is called Ulam's method \cite{MR0280310}.  After partitioning the space, data point movement from one iteration of the function provides a stochastic matrix that approximates the bistochastic matrix from Ulam's method.  Hence a Markov shift with an empirical stochastic matrix and the $1/n$ row vector approximates the dynamical system [see \cite[Chapter 1, Chapter 9]{MR2821978} for procedure and convergence rate].

We may model a stirring protocol's affect on a compression-resistant fluid with a measure-preserving dynamical system.  In this paper we are interested in discrete interations of a stirring protocol where the fluid at the beginning is the same fluid at the end.  We 'look' at the fluid before and after stirring, but not during.

Properties of an empirical Markov shift can measure and evaluate properties of a measure-preserving dynamical system \cite{ MR1668535,froyland1ulam,MR1610768}.  The second largest eigenvalue of a empirical stochastic matrix arising from Ulam's method may be used to statistically decide if a measure-preserving dynamical system is weak-mixing, ergodic, or nonergodic \cite{MR1767953, MR1384679, MR2276432,  MR0412689, MR2821978}.  To statistically test if the dynamical system is ergodic, we need to have some knowledge of
\begin{align}
P(|\widehat{\lambda}_2-1|>k:\lambda_2=1);
\end{align} to statistically test if the dynamical system is weak-mixing, we need to have some knowledge of
\begin{align}
P(|\widehat{\lambda}_2|>k:|\lambda_2|=1).
\end{align}

We use $\lambda_2$ to denote the second largest eigenvalue of a bistochastic matrix arising from Ulam's method, and $\widehat{\lambda}_2$ denotes the second largest eigenvalue of a corresponding empirical stochastic matrix (bistochastic and unistochastic matrices will be defined shortly).  The utility of $\widehat{\lambda}_2$ as a test statistic arises directly from the relationship between stochastic matrix eigenvalues and Markov shift ergodic, mixing properties.  Since the unit circle contains all eigenvalues of stochastic matrices, there are no reasonable probability distributions of $\widehat{\lambda}_2$ with $1$ as the mean or median of either $\widehat{\lambda}_2$ or $|\widehat{\lambda}_2|$.  So for hypothesis testing, we should use a probability distribution that has significant mass near $\widehat{\lambda}_2=1$ or $|\widehat{\lambda}_2|=1$.

In this paper, we show that it is reasonable to approximate the conditional probability distributions with Monte Carlo probability distributions when the equal-measure partition sets are small.  These Monte Carlo probability distributions are constructed using randomly generated Householder matrices.

Stirring protocols of compression resistant fluids, such as chocolate and water, provide examples of nearly measure-perserving dynamical systems.  The need for confidence in the mixing of food items and in the mixing of pharmaceuticals highlights the utility of such probability distributions.

We propose using randomly generated, nonzero, independent, indentically distributed real numbers to generate Householder matrices; take products of permutation matrices with Householder matrices; then square the magnitude of the products' entries  to get unistochastic matrices.  From these unistochastic matrices, construct a Monte Carlo approximation of a desired probability distribution.  From the Monte Carlo probability distribution, establish the critical value for rejecting the null-hypothesis.  The primary focus of this paper is to establish a method for determining hypothesis test critical values.  Deciding which specific probability distribution to use in a hypothesis test depends on properties of the dynamical system; we will only show that the presented Monte Carlo methods are reasonable and leave probability distribution selection for the future.

There are several ways to use Monte Carlo methods to generate bistochastic matrices; unfortunately, many techniques lead to empirical probability distributions where the central tendency of $\widehat{\lambda}_2$ is close to zero \cite[Chapter 12]{MR2821978}.  Such distributions provide little utility for a weak-mixing, ergodic, or nonergodic hypothesis test.  The methods presented here lead to Householder matrices that, in a Frobenius norm sense, are likely to be close to the identity matrix.  Squaring the magnitude of entries from these unitary matrices gives unistochatic matrices.  If we want a unistochastic matrix close to a particular permutation matrix, we may multiply the Householder matrix by the desired permutation matrix.  The main advantage of this method is that it provides probability distributions based on observed unistochastic matrices.

\section{Bistochastic Matrices}
\begin{definition}
An $n \times n$ bistochastic matrix is a stochastic matrix whose transpose is also a stochastic matrix.
\end{definition} By the Birkhoff-von Neumann theorem, the set of $n \times n$ bistochastic matrices form a convex set with permutation matrices as extreme points.  We refer to this set as Birkhoff's polytope \cite{MR2172684,MR0020547}.  Bistochastic matrices are also referred to as \emph{doubly stochastic}.
\begin{definition}
An $n \times n$ bistochastic matrix is called unistochastic if each entry is equal to the squared magnitude of some unitary matrix.
\end{definition}  The set of $n \times n$ unistochastic matrices form a proper subset of Birkhoff's polytope \cite[page 307, section 1]{MR2172684}.  Since the set of unistochastic matrices is a proper subset, the proposed method should only be used when the Ulam's method bistochastic matrix is approximately unistochastic.
\begin{definition}
An $n \times n$ Householder matrix is of the form
\begin{align}
H = I - 2\vec{v}\vec{v}^*
\end{align} where $\vec{v}$ is a unit vector.
\end{definition}  Every Householder matrix is a unitary matrix \cite[Chapter 5]{MR1417720}.  Since the set of $n \times n$ unitary matrices is closed under multiplication, taking the square magnitude of entries from a Householder matrix-permutation matrix product results in a unistochastic matrix.

\section{Modeling Dynamical Systems}

Consider running a stirring protocol on a compression resistant fluid.  Let's model this with a measure-preserving dynamical system $(\mathbb{D},\mathcal{B},\mu,f)$,
\begin{compactenum}
\item $\mathbb{D}$ represents the compression resistant fluid,
\item $\mathcal{B}$ is the Borel $\sigma$-algebra,
\item $\mu$ is rescaled Lebesgue measure so that $\mu(\mathbb{D})=1$,
\item $f:\mathbb{D}\rightarrow \mathbb{D}$ models fluid movement during stirring.
\end{compactenum}

The Monte Carlo method we will outline uses $n \times n$ unistochastic matrices arising from Householder matrix-permutation matrix products.  These unistochastic matrices are close to the permutation matrices when $n$ is large.  It is reasonable to use the described Monte Carlo distribution for hypothesis testing when the dynamical system has the following property:
For any $A,B \in \mathcal{B}$ where $P(A \cap B)=0$, if $f$ is perturbed so that
\begin{align}
P(f(x)\in A \mid x \in B) \text{ increases (decreases), }
\end{align}then for all Borel set $C \subseteq B^c$ \begin{align}
P(f(x)\in A \mid x \in C)& \text{ decreases (increases) proportionally. }
\end{align}

The $1/n$ row vector and unistochastic matrices arising from Householder matrix-permutation matrix products provide Markov shifts that reflect Ulam's method with an equal measure partition applied to such dynamical systems.  This is not saying that all such dynamical systems lead unistochastic matrices, but that the Householder constructed unistochastic matrices reflect these properties.

If $\vec{v}$ is a real unit vector and $H$ is the corresponding Householder matrix,
\begin{align} 
\vec{v} =& \begin{pmatrix}
v_1 \\
v_2 \\
\vdots \\
v_n
\end{pmatrix}, \ 
H = I - 2\vec{v}\vec{v}^t.
\end{align}  Then squaring the entries of $H$ gives a unistochastic matrix $M$,
\begin{align}
m_{ij} &= \begin{cases}
(1-2v_i^2)^2 & \text{ if } i =j, \\
4v_i^2v_j^2 & \text{ if } i \neq j.
\end{cases}
\end{align}

If $\mathbb{D}_1,\mathbb{D}_1,\ldots,\mathbb{D}_n$ are our equal measure partition sets, and $M$ arose from Ulam's method, then the entries of $M$ provide conditional probabilities,
\begin{align}
m_{ij} = P(f(x) \in \mathbb{D}_j \mid x \in \mathbb{D}_i).
\end{align}  Increasing (decreasing) $v_i^2$ leads to nearly proportional decreases (increases) in 
\begin{align}
m_{ij} = P(f(x) \in \mathbb{D}_j \mid x \in \mathbb{D}_i) \text{ when } i \neq j.
\end{align}

Because of these observations, we propose using unistochastic Ulam matrices arising from squaring the entries of real Householder matrix-permutation matrix products to model weak-mixing stirring protocols of such dynamical systems.

\section{Mixing Hypothesis Test Procedure}

In this section we will discuss and outline our procedure for testing a compression-resistant fluid stirring protocol.

There are many techniques to measure a stirring protocol's ability to mix, such as decay of correlations \cite{PhysRevA.26.717}, Fourier analysis, Artin braid patterns \cite{MR2350285,MR2317905}, chaotic advection \cite{MR2560699,MR2107644}, and other topological methods \cite{boyland2011entropy,harrington2011topological}.  Unfortunately, it is typical for different protocols to lend themselves to different analytical methods.  Thus comparing mixing quality between mechanically dissimilar stirring protocols is difficult.  The primary advantage of the Ulam method approximation is that it can be used to evaluate any incompressible fluid stirring pattern.  This allows one to compare and evaluate the mixing of stirring protocols by comparing and evaluating eigenvalues.  The main disadvantage is that the method is statistical and does not prove the results.  Another significant advantage of our method is that it requires only one iteration of stirring, in contrast to other techniques that call for iterated experiments.

Since our method only approximates the dynamical system, we make no inferences regarding strong-mixing when we conclude that the stirring protocol is weak-mixing.  If the protocol is not ergodic, then it is not weak-mixing.  If the protocol is not weak-mixing, then it is not strong-mixing.

Our test hypotheses are
\begin{compactenum}
\item $H_o: \ (\mathbb{D},\mathcal{B},\mu,f)$ is not ergodic (and hence not weak-mixing).
\item $H_{a_1}: \ (\mathbb{D},\mathcal{B},\mu,f)$ is ergodic but not weak-mixing.
\item $H_{a_2}: \ (\mathbb{D},\mathcal{B},\mu,f)$ is weak-mixing (and hence ergodic).
\end{compactenum}  

We partition the fluid into connected, equal volume regions, and use these partition sets to generate a new $\sigma$-algebra contained in $\mathcal{B}$.  If our data strongly indicate that the stirring protocol is weak-mixing or ergodic over the generated $\sigma$-algebra, we will conclude the same about the original dynamical system.  If the stirring protocol is nonmixing or nonergodic over the generated $\sigma$-algebra, then the original dynamical system is nonmixing or nonergodic.  The procedure evaluates stirring over a smaller $\sigma$-algebra, thus the test is inherently more reliable for detecting if a protocol is nonmixing or nonergodic.

The null hypothesis is that the stirring protocol is nonergodic.  It is better to reject a protocol that mixes well than to produce poorly mixed product.  The repercussions of a testing error are as follows: 
\begin{compactenum}
\item Type I Error: discard a desirable stirring protocol for a different stirring method
\item Type II Error: produce a product that is insufficiently mixed
\end{compactenum}

The stirring protocol's purpose determines the tolerable risks of error and the number of partition states.  If poorly mixed fluid could result in minor consequences or mixing on a small scale is inapt, then the number of partition regions may be relatively small.  If poorly mixed fluid could result in severe consequences, then the number of partition regions must be large and partition volume small.  For example, poorly mixed batter from a kitchen could result in unpalatable food; poorly mixed pharmaceuticals with a low LD50 could lead to overdose and death.  A mixing test for a kitchen could use a relatively coarse partition, while a pharmaceutical company would use a fine partition.

If we know an upper bound for the stirring protocol's entropy, call it $h$, then Froyland's entropy estimate and expected values show that the number of states should be greater than $e^h$ \cite{MR1668535}.

Data point movement from one iteration of stirring leads to our empirical stochastic matrix, $\widehat{P}$.  The percent of points that start in region $i$ and end in region $j$ gives us $\widehat{p}_{ij}$.  We model the entries of $\widehat{P}$ as nonindependent binomial random variables, whose probabilities come from the Ulam stochastic matrix.  Our test statistic is $\widehat{\lambda_2} = \lambda_2(\widehat{P})$.  Some dynamical systems have measure zero sets with atypical properties.  In an attempt to avoid such difficulties, we randomly select data points rather than select points from a grid.

If the data points are independent, uniform, and randomly distributed within each region, then the empirical matrix will converge to a bistochastic matrix in a Frobenius norm sense (the proof of this follows from extending a standard Monte Carlo argument \cite{robert2009introducing}).  We approximate the stirring protocol with a one-sided Markov shift, $\bigg(\big(\tfrac{1}{n},\ldots,\tfrac{1}{n}\big),\widehat{P}\bigg)$.  The pair will not define a Markov shift if \begin{align}
\big(\tfrac{1}{n},\ldots,\tfrac{1}{n}\big)\widehat{P} \neq \big(\tfrac{1}{n},\ldots,\tfrac{1}{n}\big).
\end{align} The $1/n$ row vector is a stationary distribution for any bistochastic matrix.

Since Ulam method's partitions our fluid into $n$ equal volume regions, it is reasonable to use the $1/n$ vector as the stationary distribution.  If we do not use equal volume partitions, the stationary distribution will be the probability vector corresponding to the rescaled volume of each region.  Many of the results regarding convergence, expected values, convergence rates, etc. depend on equal measure partitions, we should use equal volume regions if appropriate \cite{MR2821978}.

\begin{flushleft}
\bf{Mixing Hypothesis Test Procedure:}
\end{flushleft}
\begin{compactenum}
\item Set the type II error significance levels for both alternative hypotheses, $\alpha_1$ and $\alpha_2$.
\item Set $n$ to be the number of partition regions.
\item Decide which conditional probability distribution(s) for $|\lambda_2(P)|=1$ and $\lambda_2(P)=1$ to use.
\item Establish the critical values for $H_{a_1}$ and $H_{a_2}$, $c_1$ and $c_2$.  The purpose of this paper is to propose using Householder matrix-permutation matrix products to estimate $c_1$ and $c_2$.
\item Partition the fluid into $n$ connected equal volume regions, $\mathbb{D}_1,\mathbb{D}_2,\ldots,\mathbb{D}_n$.
\item Randomly select data points in each partition region.  These points should be independent and uniformly distributed.
\item Run the stirring protocol one time.
\item Use data point movement between regions to\\ construct an empirical stochastic matrix, $\widehat{P}$.
\item Determine the hypothesis test result.  The test statistic is $\lambda_2(\widehat{P})$;\\ compare $\mid \lambda_2(\widehat{P}) -1\mid$ to $c_1$;\\ compare $\mid \lambda_2(\widehat{P})\mid$ to $c_2$.
\item Use Froyland's entropy estimate to estimate the dynamical system's entropy.\begin{align}
-\tfrac{1}{n}\sum_{i=1}^{n}\sum_{j=1}^{n}\widehat{p}_{ij} \log \widehat{p}_{ij} 
\end{align} (we define $0 \log 0$ to equal 0).
\item If the null hypothesis is rejected in favor of weak-mixing, let the rate at which
\begin{align}
\binom{N}{n-1}(\lambda_2(\widehat{P}))^{N-n+1} \rightarrow 0 \text{ as } N \rightarrow \infty
\end{align} be our estimate of the rate of mixing.
\end{compactenum}

\section{Constructing the Monte Carlo Matrices}

Let
\begin{align}
n \in \{ 3,4,5,6,7,...\}
\end{align} be the number of equal measure states that we partition the measure-preserving dynamical system into while using Ulam's method [see \cite[Chapter 1]{MR2821978} for the procedure].  Let \begin{align}
\{u_1,u_2,\ldots,u_n\}
\end{align} be real independent, identically distributed random variables such that
\begin{align}
u_i \neq 0 \text{ almost surely, and }
E\big(\tfrac{1}{u_i^8}\big), E(u_i^8),E\big(\tfrac{1}{u_i^4}\big), E(u_i^4) < \infty.
\end{align}

Let
\begin{align}
\vec{u} = \begin{pmatrix}
u_1 \\
u_2 \\
\vdots \\
u_n
\end{pmatrix}, 
\vec{v} = \tfrac{\vec{u}}{|\vec{u}|}.
\end{align}  We may use a unit vector to construct a Householder matrix.  Let $H = (h_{ij})$ be the Householder matrix corresponding to $\vec{v}$.
\begin{align}
H &= I - 2\vec{v}\vec{v}^T 
\end{align}

The entries of $H$ are
\begin{align}
h_{ij} &= \begin{cases}
1-\tfrac{2}{u_1^2+ u_2^2+\ldots + u_n^2}u_i^2 & \text{ if } i =j, \\
-\tfrac{2}{u_1^2+ u_2^2+\ldots + u_n^2}u_i u_j & \text{ if } i \neq j.
\end{cases}
\end{align} 

Let $Q$ be a permutation matrix that we want our random unistochastic matrix to be proximal to.  Set $U = QH$.  Now, let $M= (m_{ij})$ be the matrix defined by
\begin{align}
m_{ij} &= u_{ij}^2.
\end{align}

Since $H$ is a Householder matrix and $Q$ is a permutation matrix, $U$ is a unitary matrix.  It follows that $M$ is unistochastic.

Notice that
\begin{align}
m_{ij} 
&= \begin{cases}
\big(1-\tfrac{2}{u_1^2+ u_2^2+\ldots + u_n^2}u_i^2\big)^2 & \text{ if } i =j, \\
\tfrac{4}{(u_1^2+ u_2^2+\ldots + u_n^2)^2}u_i^2 u_j^2 & \text{ if } i \neq j.
\end{cases}
\end{align}

Since $u_i \neq 0$ almost surely for all $i$, all entries of $M$ are positive almost surely; by the Perron-Frobenius theorem, all but one of $M$'s eigenvalues are of magnitude strictly less than one \cite[Chapter 8]{MR1777382}.  It follows that any Markov shift with stationary distribution
\begin{align}
\big(\tfrac{1}{n} \ldots \tfrac{1}{n}\big)
\end{align} will be strong-mixing (for Markov shifts, weak-mixing is equivalent to strong-mixing).  Our hypothesis test for weak-mixing (ergodic) requires a probability distribution over $[0,1]$ (the unit circle) with significant mass near $1$.  In the next two sections, we will see that the expected value of $M$'s eigenvalues converge to one as $n$ goes to infinity.

How does our dynamical system relate to $n$?  Generally speaking, finer partitions are more apt to detect nonmixing (nonergodicity).  If we are confident in mixing, we will use a coarse partition to reduce effort; if our confidence in mixing is poor, we will use a fine partition.

By the Birckoff-von Neumann theorem, bistochastic matrices are convex combinations of permutation matrices \cite{MR2172684,MR0020547}.  So our unistochastic matrices will tend to be near the 'corners' of the set of bistochastic matrices.  For any statistic from unistochastic matrices we are interested in, we may use such matrices to generate a Monte Carlo empirical probability distribution.

\section{Establishing Critical Values}

In this section, we will outline the procedure we propose for establishing critical values for a weak-mixing, ergodic, nonergodic hypothesis test.

Our test hypotheses are
\begin{compactenum}
\item $H_o: \ (\mathbb{D},\mathcal{B},\mu,f)$ is not ergodic (and hence not weak-mixing).
\item $H_{a_1}: \ (\mathbb{D},\mathcal{B},\mu,f)$ is ergodic but not weak-mixing.
\item $H_{a_2}: \ (\mathbb{D},\mathcal{B},\mu,f)$ is weak-mixing (and hence ergodic).
\end{compactenum}
After partitioning the space into $n$ equal measure connected subsets, Ulam's method approximates the dynamical system with a Markov shift.  We will approximate the bistochastic matrix defining the Markov shift's measure, $P$, with an empirical stochastic matrix, $\widehat{P}$.  So we approximate 
\begin{align}
(\mathbb{D},\mathcal{B},\mu,f) \text{ with } \Big(\big(\tfrac{1}{n},\ldots,\tfrac{1}{n}\big),\widehat{P}\Big).
\end{align}

\begin{remark}
The pair
\begin{align}
\Big(\big(\tfrac{1}{n},\ldots,\tfrac{1}{n}\big),\widehat{P}\Big)
\end{align} will not define a Markov shift if
\begin{align}
\big(\tfrac{1}{n},\ldots,\tfrac{1}{n}\big)\widehat{P} \neq \big(\tfrac{1}{n},\ldots,\tfrac{1}{n}\big),
\end{align} but if our data points are uniform random variables within each state, then
\begin{align}
E(\Vert P - \widehat{P} \Vert_F) \rightarrow 0
\end{align} as the minimum number of points in a state goes towards infinity. It follows that for each eigenvalue
\begin{align}
\mid \lambda_i(P) - \lambda_i(\widehat{P}) \mid \rightarrow 0 \text{ in probability } \ \forall i
\end{align} in the Hausdorf topology when our data points are uniform random variables within each state and the minimum number of points in a state goes towards infinity \cite[Chapter 8]{MR2821978}.
\end{remark}

Our test statistic is the second largest eigenvalue of $\widehat{P}$.  Let $\alpha_1, \alpha_2 \in (0,1)$ be the alpha values for the hypothesis test; let $c_1, c_2 \in (0,1)$ be the corresponding critical values, $c_2 \leq 1-c_1$,
\begin{align}
P(\mid \lambda_2(\widehat{P}) - 1 \mid \geq c_1 : \lambda_2(P)=1)& < \alpha_1, \\
P(\mid \lambda_2(\widehat{P}) \mid \leq c_2 : \mid \lambda_2(P) \mid =1)& < \alpha_2.
\end{align}

Our goal is to use Householder matrices to estimate $c_1$ and $c_2$.  The probability distribution used to establish $c_1,c_2$ should reflect properties of a class of dynamical systems containing our stirring protocol.

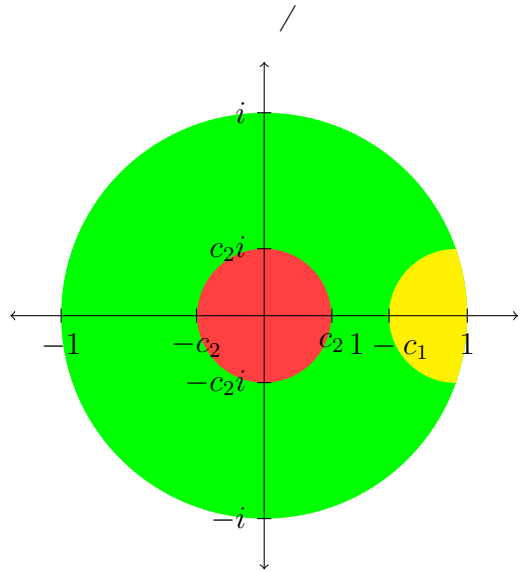
\begin{tikzpicture}[scale=2.7]
\node[below, text width=12cm] at (0,-1.25) {
\small\textbf{Argand diagram of mixing hypothesis test criteria:}\par
\begin{compactenum}
\item If $\lambda_2(\widehat{P})$ is in the region containing $1$ (yellow), fail to reject the null hypothesis; conclude that the dynamical system is nonergodic.\\
\item If $\lambda_2(\widehat{P})$ is in the outer region away from $1$ (green), reject the null hypothesis in favor of the first alternative hypothesis; conclude that the dynamical system is ergodic, but not weak-mixing.\\
\item If $\lambda_2(\widehat{P})$ is in the center region (red), reject the null hypothesis in favor of the second alternative hypothesis; conclude that the dynamical system is weak-mixing (and hence ergodic).
\end{compactenum}
};
\fill[fill = green] (0,0) circle (1);
\fill[fill = yellow] (0.944,0.329) arc (90:270:0.33) -- (0.944,-0.329) -- (0.95,-0.312) -- (0.96,-0.28) -- (0.97,-0.243) -- (0.98,-0.199) -- (0.99,-0.141) -- (0.995,-0.0999) -- (0.999,-0.045) -- (1,0) -- (1,0.045) -- (0.995,0.0999) -- (0.99,0.141) -- (0.98,0.199) -- (0.97,0.243) -- (0.96,0.28) -- (0.95,0.312) -- (0.944,0.329);
\fill[fill = red!75] (0,0) circle (0.33);
\draw[<->] (-1.25,0) -- (1.25,0);
\foreach \x/\xtext in {-1, -0.333/-c_2,0.333/c_2, 0.615/1-c_1, 1}
\draw (\x cm,1pt) -- (\x cm,-1pt) node[anchor=north] {$\xtext$};
\draw[<->] (0,-1.25) -- (0,1.25);
\foreach \y/\ytext in {-1/-i, -0.33/-c_2 i,0.33/c_2 i, 1/i}
\draw (1pt,\y cm) -- (-1pt,\y cm) node[anchor=east] {$\ytext$};
\end{tikzpicture}

\begin{flushleft}
\bf{Establishing Critical Values for the Test:}
\end{flushleft}
\begin{compactenum}
\item Partition $\mathbb{D}$ into $n$ equal measure connected subsets.  If an upper bound of the dynamical system's entropy is known, call the upper bound $h$, set $n$ greater than $e^h$ \cite{MR1668535}.
\item Select a random variable with which to construct unit vectors.\\ Let $u_1,u_2,\ldots,u_n$ be independent, identically distributed, random variables,\\ $\vec{v}_i = \tfrac{\vec{u}}{\Vert u \Vert_2}$.
\item Set $N \in \mathbb{N}$ so that our empirical probability distributions will be sufficiently accurate.
\item Select permutation matrices $\{Q_i\}_{i=1}^N$ near which we want the probability distribution to have significant mass.
\item Randomly generate $N$ Householder matrices, $H_i = I-2\vec{v}_i\vec{v}_i^T$.  Then square the entries of $Q_iH_i$ to get the matrix $M_i$.
\item Use $\{ \lambda_2(M_i) \}_{i=1}^N$ to approximate $P(\mid \lambda_2(\widehat{P}) - 1 \mid \geq k : \lambda_2(P)=1)$.  Use the approximation to estimate $c_1$.
\item Use $\{ \mid \lambda_2(M_i) \mid \}_{i=1}^N$ to approximate $P(\mid \lambda_2(\widehat{P}) \mid < k : \lambda_2(P)=1)$.  Use the approximation to estimate $c_2$.
\end{compactenum}

\section{Matrix Convergence}

In this section, we will show that $M_i$ from the previous section will converge to the permutation matrix $Q_i$ as $n$ increases.  Since permutation matrix eigenvalues are on the unit circle, as a random variable, it is likely that the second largest eigenvalue from one of our unistochastic matrices will be near magnitude one.  Because of the likely proximity to one, it is reasonable to use a probability distribution from such an eigenvalue to establish critical values for our weak-mixing, ergodic, nonergodic hypothesis test.

Our proofs take advantage of the Frobenius norm.  After using Jensen's inequality to remove the square root from consideration, finding an expected value upper bound is similar to finding a second moment.  A permutation matrix acting on a matrix does not change the magnitude of the entries, without loss of generality will prove the results for when $Q$ is the indentity matrix and focus on $M$'s convergence to the identity matrix.

\begin{proposition}
If $M$ is a matrix constructed in section $2$ with $Q = I$, $n \in \{ 3,4,5,\ldots \}, 
\{ u_i \}_{i=1}^{n}$ are identically distributed, $u_i \neq 0$ a.s. and
\begin{align}
E(u_i^4),E(u_i^8),E\big(\tfrac{1}{u_i^4}\big),E\big(\tfrac{1}{u_i^8}\big)< \infty,
\end{align}
then $E(\parallel M-I \parallel_F) \rightarrow 0$ as $n \rightarrow \infty$.  Moreover,
\begin{align}
E(\parallel M -I \parallel_F^2) \leq&  \tfrac{16n}{(n-1)^4} E(\tfrac{1}{u_i^8}) E(u_i^8) \\
&+\tfrac{16n}{(n-1)^2} E(\tfrac{1}{u_i^4}) E(u_i^4) \\ 
&+\tfrac{16n(n-1)}{(n-2)^4}  E(\tfrac{1}{u_i^8}) (E(u_i^4))^2.\\
\end{align}
\end{proposition}

\begin{proof}
First, by Jensen's inequality
\begin{align}
E(\parallel M-I \parallel_F) \leq \sqrt{E(\parallel M-I \parallel_F^2)}.
\end{align}
So it is sufficient to show the second part of the proposition.  Let's look at the entries of $M-I$; by computation we see that:
\begin{align}
(M-I)_{ij} 
&= \begin{cases}
-4(\tfrac{1}{u_1^2+ u_2^2+\ldots + u_n^2}u_i^2) & \\
 \ \ \times (1-\tfrac{1}{u_1^2+ u_2^2+\ldots + u_n^2}u_i^2 )  & \text{ if } i =j, \\
\tfrac{4}{(u_1^2+ u_2^2+\ldots + u_n^2)^2}u_i^2 u_j^2 & \text{ if } i \neq j.
\end{cases}
\end{align}
It follows that 
\begin{align}
(M-I)_{ij}^2 
&= \begin{cases}
\tfrac{16}{(u_1^2+ u_2^2+\ldots + u_n^2)^4}u_i^8 & \\
\ \ -\tfrac{32}{(u_1^2+ u_2^2+\ldots + u_n^2)^3}u_i^6 & \\ 
\ \ +\tfrac{16}{(u_1^2+ u_2^2+\ldots + u_n^2)^2}u_i^4  & \text{ if } i =j, \\
\tfrac{16}{(u_1^2+ u_2^2+\ldots + u_n^2)^4}u_i^4 u_j^4 & \text{ if } i \neq j.
\end{cases}
\end{align}
If we expand the addends and remove the negative terms, it follows that almost surely
\begin{align}
\parallel M -I \parallel_F^2 <& 
\sum_{i=1}^{n}\tfrac{16}{(u_1^2+ u_2^2+\ldots + u_n^2)^4}u_i^8 \\
&+\sum_{i=1}^{n}\tfrac{16}{(u_1^2+ u_2^2+\ldots + u_n^2)^2}u_i^4  \\
& +\sum_{i \neq j} \tfrac{16}{(u_1^2+ u_2^2+\ldots + u_n^2)^4}u_i^4 u_j^4.
\end{align}
Almost surely, all of the terms in the denominators are positive; if we subtract terms from the denominators, we get an upper bound on the fractions.  Thus almost surely
\begin{align}
\parallel M -I \parallel_F^2 <& 
\sum_{i=1}^{n}\tfrac{16}{(u_1^2+ u_2^2+\ldots + u_n^2 -u_i^2)^4}u_i^8 \\
&+\sum_{i=1}^{n}\tfrac{16}{(u_1^2+ u_2^2+\ldots + u_n^2-u_i^2)^2}u_i^4  \\
& +\sum_{i \neq j} \tfrac{16}{(u_1^2+ u_2^2+\ldots + u_n^2-u_i^2-u_j^2)^4}u_i^4 u_j^4.
\end{align}
Notice that the numerator and denominator in each fraction in the upper bound are independent.


The subtraction of terms in denominators removed some positive terms, so the denominators are sums of positive terms.  Therefore, we may use the harmonic-arithmetic means inequality,
\begin{align}
\parallel M -I \parallel_F^2 <& 16\sum_{i=1}^{n}\tfrac{1}{(n-1)^8}(\tfrac{1}{u_1^2}+\tfrac{1}{u_2^2}+\ldots+\tfrac{1}{u_n^2}-\tfrac{1}{u_i^2})^4u_i^8\\
&+16\sum_{i=1}^{n}\tfrac{1}{(n-1)^4}(\tfrac{1}{u_1^2}+\tfrac{1}{u_2^2}+\ldots+\tfrac{1}{u_n^2}-\tfrac{1}{u_i^2})^2u_i^4\\
&+16\sum_{i \neq j} \tfrac{1}{(n-2)^8}(\tfrac{1}{u_1^2}+\tfrac{1}{u_2^2}+\ldots+\tfrac{1}{u_n^2}-\tfrac{1}{u_i^2}-\tfrac{1}{u_j^2})^4u_i^4 u_j^4
\end{align}
Now let's take expected values; since the $u_i$'s are independent,
\begin{align}
\tfrac{E(\parallel M -I \parallel_F^2)}{16} \leq& \sum_{i=1}^{n}E((\tfrac{1}{u_1^2}+\tfrac{1}{u_2^2}+\ldots+\tfrac{1}{u_n^2}-\tfrac{1}{u_i^2})^4)\tfrac{E(u_i^8)}{(n-1)^8}\\ 
&+\sum_{i=1}^{n}E((\tfrac{1}{u_1^2}+\tfrac{1}{u_2^2}+\ldots+\tfrac{1}{u_n^2}-\tfrac{1}{u_i^2})^2)\tfrac{E(u_i^4)}{(n-1)^4}  \\
&+\sum_{i \neq j} E((\tfrac{1}{u_1^2}+\tfrac{1}{u_2^2}+\ldots+\tfrac{1}{u_n^2}-\tfrac{1}{u_i^2}-\tfrac{1}{u_j^2})^4)\tfrac{E(u_i^4)E( u_j^4)}{(n-2)^8}.
\end{align}
Next we use Minkowski's inequality and the fact that the $u_i$'s are independently distributed,
\begin{align}
E(\parallel M -I \parallel_F^2) \leq& 
\tfrac{16n}{(n-1)^4} E(\tfrac{1}{u_i^8}) E(u_i^8)\\ 
& +\tfrac{16n}{(n-1)^2} E(\tfrac{1}{u_i^4}) E(u_i^4)  \\
& +\tfrac{16n(n-1)}{(n-2)^4}  E(\tfrac{1}{u_i^8}) (E(u_i^4))^2.\\
\end{align}
Since $E(\tfrac{1}{u_i^8}), E(u_i^8),E(\tfrac{1}{u_i^4})$, and $E(u_i^4)$ are all finite, it follows that 
\begin{align}
E(\parallel M -I \parallel_F^2) \rightarrow 0
\end{align}
as $n \rightarrow \infty$.
Hence by Jensen's inequality,
\begin{align}
E(\parallel M-I \parallel_F) \rightarrow 0
\end{align}
as $n \rightarrow \infty$.
\end{proof}







Since the second largest eigenvalue gives a test statistic to decide if a measure-preserving dynamical system is weak mixing, ergodic, or nonergodic, we need a conditional probability distribution of eigenvalues to conduct hypothesis tests.  To statistically test if a measure-preserving dynamical system is weak-mixing, we could randomly select permutation matrices 
\begin{align}
\{ Q_k \}_{k=1}^N
\end{align}
and generate $\{ M_k  \}_{k=1}^N$, with our Householder method, then use the empirical probability distribution from 
\begin{align}
\{ |\lambda_2(M_k)| \}_{i=1}^N
\end{align}
to establish the critical value for the weak-mixing hypothesis test.

To statistically test if a measure-preserving dynamical system is ergodic, we could randomly select permutation matrices 
\begin{align}
\{ Q_k : \text{the multiplicity of } \lambda=1 \text{ is at least two}\}_{k=1}^N
\end{align} and use Householder matrices to generate $\{ M_k  \}_{k=1}^N$, then use the empirical probability distribution from 
\begin{align}
\{ \lambda_2(M_k) \}_{k=1}^N
\end{align}
to establish a critical value for the ergodic hypothesis test.

\section{Using Specific Random Variables}

In this section, we find more precise upper bounds for specific random variables.  These upper bounds give better estimates of convergence rate than the results in the previous section.  The first two proofs in this section start out the same way as the first proof in the previous section, then the arguments take advantage of the distribution properties.

Let's find a more precise upper bound when the $u_i$'s are independent standard normal random variables.  The proof is similar to the first convergence proof, the difference is that we take advantage of the relationship between normal random variables and $\chi^2$-distributions.

\begin{proposition}
If the $u_i$'s in the construction of a unistochastic matrix are independent standard normal random variables and $n \in \{11,12,13,\ldots\}$, then
\begin{align}
E(\parallel M -I \parallel_F^2) <& \tfrac{1680n}{(n-3)(n-5)(n-7)(n-9)} \\
&+\tfrac{48n}{(n-3)(n-5)} \\
&+\tfrac{144n(n-1)}{(n-4)(n-6)(n-8)(n-10)}.
\end{align}
\end{proposition}

\begin{proof}
From the previous proof, we know that almost surely
\begin{align}
\parallel M -I \parallel_F^2 <& 
\sum_{i=1}^{n}\tfrac{16}{(u_1^2+ u_2^2+\ldots + u_n^2 -u_i^2)^4}u_i^8 \\
&+\sum_{i=1}^{n}\tfrac{16}{(u_1^2+ u_2^2+\ldots + u_n^2-u_i^2)^2}u_i^4  \\
& +\sum_{i \neq j} \tfrac{16}{(u_1^2+ u_2^2+\ldots + u_n^2-u_i^2-u_j^2)^4}u_i^4 u_j^4.
\end{align}

Since $u_i$'s are independent standard normal random variables, we may replace the $u_i$'s with $\chi^2$-random variables when we take expected values.
 
\begin{align}
E(\parallel M -I \parallel_F^2) < 
&\sum_{i=1}^{n}16E(\tfrac{1}{(\gamma_{n-1})^4})E(u_i^8) \\
&+\sum_{i=1}^{n}16E(\tfrac{1}{(\gamma_{n-1})^2})E(u_i^4)  \\
& +\sum_{i \neq j} 16E(\tfrac{1}{(\gamma_{n-2})^4})E(u_i^4)E(u_j^4).
\end{align}

We use $\gamma_j$ to denote a $\chi^2$-random variable with $j$ degrees of freedom.  If we take expected values and remove negative terms, it follows that
\begin{align}
E(\parallel M -I \parallel_F^2) <& \tfrac{1680n}{(n-3)(n-5)(n-7)(n-9)}\\
&+\tfrac{48n}{(n-3)(n-5)}\\
&+\tfrac{144n(n-1)}{(n-4)(n-6)(n-8)(n-10)}.
\end{align}
\end{proof}

Now let's consider gamma random variables.  A finite sum of independent gamma random variables with the same scale parameter is a new gamma random variable with the same scale parameter, but the shape parameter is the sum of the addend shape parameters.  In the next proof, we look at independent and indentically distributed gamma random variables.

\begin{proposition}
If the $u_i$'s in the construction of our unistochastic matrix are independent $\Gamma(\alpha,\beta)$ random variables and $\tfrac{8}{\alpha}+2 < n $, then
\begin{align}
E(\parallel M -I \parallel_F^2) <&
\tfrac{16n^5 \prod_{i=0}^{7}(\alpha+i)}{ \prod_{i=1}^8[(n-1)\alpha-i]} \\    
&+ \tfrac{16n^3 \prod_{i=0}^{3}(\alpha+i)  }{ \prod_{i=1}^4[(n-1)\alpha-i]} \\         
&+  \tfrac{16n^5(n-1) \prod_{i=0}^{3}(\alpha+i)^2 }{ \prod_{i=1}^8[(n-2)\alpha - i]}.
\end{align}
\end{proposition}

\begin{proof}
Previously we showed that almost surely
\begin{align}
\parallel M -I \parallel_F^2 <& 
\sum_{i=1}^{n}\tfrac{16}{(u_1^2+ u_2^2+\ldots + u_n^2 -u_i^2)^4}u_i^8 \\
&+\sum_{i=1}^{n}\tfrac{16}{(u_1^2+ u_2^2+\ldots + u_n^2-u_i^2)^2}u_i^4  \\
& +\sum_{i \neq j} \tfrac{16}{(u_1^2+ u_2^2+\ldots + u_n^2-u_i^2-u_j^2)^4}u_i^4 u_j^4.
\end{align}
Using the Cauchy-Schwarz inequality and the fact that $0< u_i$ almost surely for all $i$, we get
\begin{align}
\parallel M -I \parallel_F^2 <& 
\sum_{i=1}^{n}\tfrac{16n^4}{(u_1+ u_2+\ldots + u_n -u_i)^8}u_i^8 \\
&+\sum_{i=1}^{n}\tfrac{16n^2}{(u_1+ u_2+\ldots + u_n-u_i)^4}u_i^4  \\
& +\sum_{i \neq j} \tfrac{16n^4}{(u_1+ u_2+\ldots + u_n-u_i-u_j)^8}u_i^4 u_j^4.
\end{align}
If we take expected values, and take advantage of the independent and identically distributed $u_i$'s,

\begin{align}
E(\parallel M -I \parallel_F^2) <& 
E(\tfrac{16n^5}{(u_1+ u_2+\ldots + u_n -u_i)^8})E(u_i^8) \\
&+E(\tfrac{16n^3}{(u_1+ u_2+\ldots + u_n-u_i)^4})E(u_i^4)  \\
& +E(\tfrac{16n^5(n-1)}{(u_1+ u_2+\ldots + u_n-u_i-u_j)^8})E(u_i^4)E( u_j^4).
\end{align}
Since the $u_i$'s are gamma random variables and $\tfrac{8}{\alpha}+2 < n $, we may replace the denominator sums with gamma random variables.  Let $\gamma_{j}$ denote a gamma random varible with parameters $j\alpha$, and $\beta$.
\begin{align}
E(\parallel M -I \parallel_F^2) <&
\tfrac{16n^5 \prod_{i=0}^{7}(\alpha+i)}{ \prod_{i=1}^8[(n-1)\alpha-i]} \\
&+ \tfrac{16n^3 \prod_{i=0}^{3}(\alpha+i)  }{ \prod_{i=1}^4[(n-1)\alpha-i]} \\
&+  \tfrac{16n^5(n-1) \prod_{i=0}^{3}(\alpha+i)^2 }{ \prod_{i=1}^8[(n-2)\alpha - i]}.
\end{align}
\end{proof}

Next we will consider bistochastic matrices arising from random Householder matrices where 
\begin{align}
u_1 = u_2 = \ldots = u_n , \text{ and } u_i \neq 0
\end{align} almost surely.  It follows that almost surely
\begin{align}
H = I - \tfrac{2}{n}\begin{bmatrix}1 & \ldots & 1 \\ \vdots & \ddots & \vdots \\ 1 & \ldots & 1\end{bmatrix},
M = \big(\tfrac{n-4}{n}\big)I + \begin{bmatrix}\tfrac{4}{n^2} & \ldots & \tfrac{4}{n^2} \\ \vdots & \ddots & \vdots \\ \tfrac{4}{n^2} & \ldots & \tfrac{4}{n^2} \end{bmatrix}.
\end{align}
 
For the previous matrix probability distributions described, we use random numbers to build matrices, then use the matrices to construct an empirical approximation to the probability distribution of the statistic under consideration.  When $u_1=u_2=\ldots = u_n$, $u_i \neq 0$ almost surely, the aspects of the matrix only depends on $n$, we know the precise distribution of all matrix statistics.  If reasonable for the dynamical system under consideration, we could use $M$ to establish critical values for hypothesis testing.
So let's look at some aspects of such $M$'s, after a lemma.
\begin{lemma}
If $D_n$ and $S_n$ are $n \times n$ symmetric matrices such that
\begin{align}
D_n = \begin{bmatrix}
\alpha & \beta & \beta  & \ldots & \beta \\
\beta & \alpha & \beta  & \ldots & \beta \\
\beta & \beta & \alpha  & \ldots & \beta \\
\vdots& \vdots& \vdots  & \ddots & \vdots\\
\beta & \beta & \beta & \ldots  & \alpha
\end{bmatrix} \text{ and } 
S_n = \begin{bmatrix}
\beta & \vec{\beta}^T  \\
\vec{\beta} & D_{n-1}
\end{bmatrix}
\end{align}
where $\vec{\beta}$ is the $n-1$ vector with $\beta$ for all entries, then
\begin{align}
\det(D_n) &= (\alpha - \beta)^{n-1}(\alpha + (n-1)\beta) \text{ and } \\
\det(S_n) &= (\alpha - \beta)^{n-1}\beta.
\end{align}
\end{lemma}

\begin{proof}[Proof by induction:]

Assume true for $n$.  Using the fact that interchanging any two rows or any two columns of a real matrix changes the sign of the determinant, we see that
\begin{align}
\det(D_{n+1}) =& \alpha \det(D_{n}) -\beta n \det(S_{n}) \text{ and }\\
\det(S_{n+1}) =& \beta \det(D_{n}) -\beta n \det(S_{n}).
\end{align}
Using the induction hypothesis, these equations become
\begin{align}
\det(D_{n+1}) =& \alpha \bigg((\alpha - \beta)^{n-1}(\alpha + (n-1)\beta)\bigg) -\beta n \bigg((\alpha - \beta)^{n-1}\beta \bigg) \text{ and }\\
\det(S_{n+1}) =& \beta  \bigg((\alpha - \beta)^{n-1}(\alpha + (n-1)\beta)\bigg) -\beta n \bigg((\alpha - \beta)^{n-1}\beta \bigg).
\end{align}

Factoring out the $(\alpha - \beta)$ terms gives us the results.
\end{proof}

\begin{proposition}
If $M$ is the $n \times n$ matrix matrix
\begin{align}
M = \Big(\tfrac{n-4}{n}\Big)I + \begin{bmatrix}\tfrac{4}{n^2} & \ldots & \tfrac{4}{n^2} \\ \vdots & \ddots & \vdots \\ \tfrac{4}{n^2} & \ldots & \tfrac{4}{n^2} \end{bmatrix},
\end{align}
then $\det(M) = \big( \tfrac{n-4}{n} \big)^{n-1}$, $trace(M) = \tfrac{(n-2)^2}{n}$ and the Jordan canonical form of $M$ is the diagonal matrix with entries $
(1, \tfrac{n-4}{n},\ldots,\tfrac{n-4}{n})$.
\end{proposition}

\begin{proof}
The trace of $M$ follows from the definition.  The determinant follows from the previous lemma by setting $\alpha = \tfrac{(n-2)^2}{n^2}$, and $\beta = \tfrac{4}{n^2}$.

Now, the matrix is a symmetric real matrix; hence it is diagonalizable.  The eigenvalues follow from the previous lemma by setting $\alpha = \tfrac{(n-2)^2}{n^2} -\lambda$, and $\beta = \tfrac{4}{n^2}$ to get the characteristic polynomial of $M$.
\end{proof}

Since $M$ is diagonalizable, if the Markov shift $((1/n,1/n,\ldots,1/n),M)$ approximates a measure-preserving dynamical system that is mixing, the estimate of mixing rate is the rate at which 
\begin{align}
\big( \tfrac{n-4}{n} \big)^N \rightarrow 0 \text{ as } N \rightarrow \infty, 
\end{align} instead of the estimate given in \cite[Chapter 4]{MR2821978}, $\binom{N}{n-1}( \tfrac{n-4}{n} )^{N-n+1} \rightarrow 0$ as $N \rightarrow \infty$.

\section{Two Region Partitions}

There are few instances of interest where one would use our method with two partition regions.  We look at this special case as an example to help develop understanding.  A potential application is equal ratio mixing of items with minimal consequences of poor mixing, such as combining blends of coffee.  When combining two equal volumes of coffee, poor mixing would result is inconsistent taste.  Only the most serious baristas would say that inconsistent cup-of-Joe flavor is worse than poorly mixing pharmaceuticals.

Say that our unit vector is
\begin{align}
\vec{v} = \begin{pmatrix} v_1 \\ v_2 \end{pmatrix}.
\end{align}
Since the vector has norm one, we may write its Householder matrix as
\begin{align}
H = \begin{bmatrix}
1-2v_1^2 & \pm2v_1\sqrt{1-v_1^2} \\
\pm2v_1\sqrt{1-v_1^2} & 2v_1^2-1
\end{bmatrix}.
\end{align}  There are two possible doubly-stochastic matrices arising from our method,
\begin{align}
\begin{bmatrix}
(1-2v_1^2)^2 & 4v_1^2(1-v_1^2) \\
4v_1^2(1-v_1^2) & (1-2v_1^2)^2\\
\end{bmatrix}, \text{ and } \begin{bmatrix}
4v_1^2(1-v_1^2) & (1-2v_1^2)^2   \\
(1-2v_1^2)^2 & 4v_1^2(1-v_1^2)  \\
\end{bmatrix}.
\end{align}  Whose characteristic polynomials are
\begin{align}
(1-\lambda)(8v_1^4-8v_1^2+1-\lambda), \text{ and } (1-\lambda)(-8v_1^4+8v_1^2-1-\lambda).
\end{align} The second largest eigenvalue depends on $v_1$.  Let's graph the relationship.

\begin{tikzpicture}[scale=3]
\draw[blue,domain=-1:1,smooth] plot (\x,{8*(\x)^4-8*(\x)^2+1});
\draw[purple,domain=-1:1,smooth] plot (\x,{-8*(\x)^4+8*(\x)^2-1});
\draw[<->] (-1.25,0) -- (1.25,0);
\foreach \x/\xtext in {-1/-1, -0.707/-\tfrac{\sqrt{2}}{2}, -0.408/-\tfrac{\sqrt{6}}{6}, 0.408/\tfrac{\sqrt{6}}{6}, 0.707/\tfrac{\sqrt{2}}{2}, 1/1}
\draw (\x cm,1pt) -- (\x cm,-1pt) node[anchor=north] {$\xtext$};
\draw[<->] (0,-1.25) -- (0,1.25);
\foreach \y/\ytext in {-1/-1, -0.111/-\tfrac{1}{9}, 0.111/\tfrac{1}{9}, 1/1}
\draw (1pt,\y cm) -- (-1pt,\y cm) node[anchor=east] {$\ytext$};
\draw (1.3,-0.1) node {$v_1$};
\draw (0.1,1.2) node {$\lambda_2$};
\node[below, text width=11.3cm] at (0.35,-1.25) {
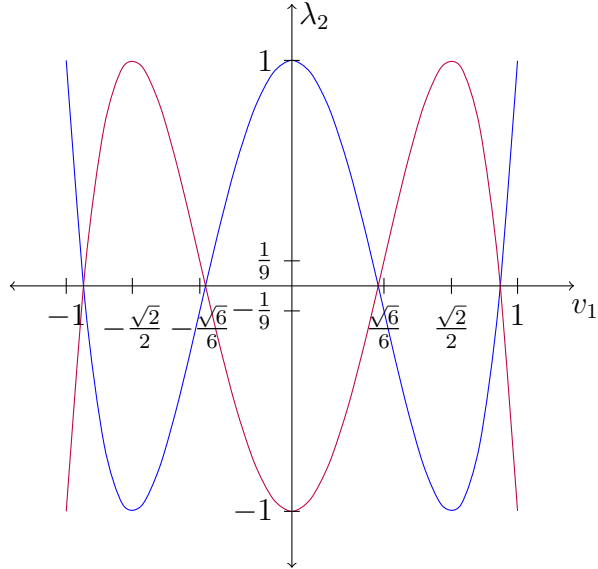
\small\textbf{Graphs of the relationship between $v_1$ and $\lambda_2$ when $n=2$}};
\end{tikzpicture}

When $n=2$, the relationship between $v_1$ and second largest eigenvalue defines a function from $[-1,1]$ to $[-1,1]$. In this situation, it is feasible to compute
\begin{align}
P(|\widehat{\lambda}_2-1|>k:\lambda_2=1) \text{ and }
P(|\widehat{\lambda}_2|>k:|\lambda_2|=1).
\end{align}

If $v_1$ is a beta random variable with parameters $\alpha$ and $\beta$, then
\begin{align}
E(\lambda_2) =& \pm\tfrac{8\alpha(\alpha+1)(\alpha+2)(\alpha+3)}{(\alpha+\beta)(\alpha+\beta+1)(\alpha+\beta+2)(\alpha+\beta+3)}\\ & \ \ \ \mp
\tfrac{8\alpha(\alpha+1)}{(\alpha+\beta)(\alpha+\beta+1)} \pm 1,
\end{align} where the sign of addends depends on the permutation matrix used.


\end{document}